\documentclass[12pt,a4paper,reqno]{amsart}
\usepackage{xypic,a4wide,amsmath,amssymb,amsthm}
\usepackage[mathcal]{euscript}
\usepackage{mathrsfs}
\let\mathcal=\mathscr

\newtheorem{thm}{Theorem}[section]
\newtheorem{corol}[thm]{Corollary}
\newtheorem{lemma}[thm]{Lemma}
\newtheorem{prop}[thm]{Proposition}

\theoremstyle{remark}
\newtheorem{rem}[thm]{Remark}
\newtheorem{ex}[thm]{Example}
\newenvironment{remark}{\begin{rem}\rm}{\qee\end{rem}}

\newcommand{\End}{\mbox{\it End}}
\newcommand{\cE}{{\mathcal E}}
\newcommand{\cO}{{\mathcal O}}

\newcommand{\CC}{{\mathcal C}^\infty}

\newcommand{\Ext}{\mbox{\rm Ext}}

\newcommand{\ati}{{{\mathcal D}_{\cE}}}
\newcommand{\Lie}{\hbox{\em\char'44}\!}

\newcommand{\C}{{\mathbb C}}

\newcommand{\Q}{{\mathfrak Q}}

\newcommand{\bbL}{{\mathbb L}}

\newcommand{\qee}{\mbox{\hspace{0.2mm}}\hfill$\triangle$}

\newcommand{\ccL}{{\mathcal L}}
\newcommand{\tr}{\operatorname{tr}}
\begin{document}
\begin{center}
{\Large\bf On localization in holomorphic \\[5pt] equivariant cohomology} \\[10pt]
{\sc U. Bruzzo} \\ Scuola Internazionale Superiore di Studi Avanzati, \\ Via Bonomea 265, 34136 Trieste, Italy;  \\
 Istituto Nazionale di Fisica Nucleare, Sezione di Trieste\\ E-Mail: {\tt bruzzo@sissa.it} \\[10pt]
{\sc V. Rubtsov} \\
Universit\'e d'Angers, D\'epartement de Math\'ematiques, \\
UFR Sciences, LAREMA, UMR 6093 du CNRS,\\
2 bd.~Lavoisier, 49045 Angers Cedex 01, France; \\
ITEP Theoretical Division, 25 Bol.~Tcheremushkinskaya,\\ 117259, Moscow, Russia
\\ E-mail: {\tt Volodya.Roubtsov@univ-angers.fr}
\end{center} \vfill
\begin{quote} \small  {\sc Abstract.}
We study a holomorphic equivariant cohomology built out of the Atiyah algebroid of an equivariant holomorphic vector bundle and prove a related localization formula. This encompasses various residue formulas in complex geometry, in particular we shall show that it contains as   special cases Carrell-Liebermann's and Feng-Ma's residue formulas, and Baum-Bott's formula for the zeroes of a meromorphic vector field.

\end{quote}
\par

\vfill
\parbox{.75\textwidth}{\hrulefill}\par
\begin{minipage}[c]{\textwidth} \small 
 \noindent   
Date: \today  \par\noindent
The authors gratefully acknowledge
financial support and hospitality during the respective visits to
Universit\'e d'Angers and SISSA. Support for this work was provided by {\sc misgam,} by   the I.N.F.N.~project PI14 ``Nonperturbative dynamics of gauge theories", the {\sc einstein} Italo-Russian project  ``Integrability in topological string and field theory,''  the {\sc matpyl}   Angers-SISSA project
``Lie algebroids, equivariant cohomology, and topological quantum field and string theories,'' and the {\sc geanpyl} programme.
\end{minipage}

\newpage

\section{Introduction}
 
In his fundamental paper \cite{Bott67}, Bott wrote formulas allowing to express
the Chern numbers of a compact manifold $X$ in terms of the zeroes of a vector field
on $X$. This was done in two cases, for a holomorphic vector field (when $X$ is a complex manifold), and for a vector field preserving a Riemannian metric on $X$. The second case was later better understood in the setting of equivariant cohomology \cite{AtiBott}: one can introduce an equivariant de Rham complex, and the integral of any cocycle can be expressed in terms of its restrictions to the fixed points of the group action (which is an $S^1$ action for the case of a vector field preserving a Riemannian metric).

The analogy between the two cases, the holomorphic and the Riemannian one, is not complete. However, it is possible to define some kind of ``formally equivariant'' Dolbeault complex and derive from it localization formulas, see e.g.~\cite{Liu}. Our goal in this paper is to propose
a comprehensive approach to holomorphic equivariant cohomology. We start from
the holomorphic
Atiyah algebroids $\ati$ associated to a holomorphic vector bundle $\cE$  on a complex manifold $X$, and assuming that there is a vector field on $X$ that may be lifted to a section of $\ati$, we
introduce a holomorphic equivariant complex and prove a related localization formula (Section
\ref{atiyahsec}). When $\cE=\{0\}$ this   cohomology essentially reduces to K.~Liu's
holomorphic equivariant cohomology \cite{Liu}. 
By considering a twist by an effective line bundle, we can also prove a general localization formula related to the action of a meromorphic vector field.

Our localization formula contains the Baum-Bott   formula \cite{BaumBott}, Feng-Ma's localization formula \cite{Feng-Ma05},  
and  the Carrell-Liebermann residue formula \cite{CL,Carr} as special cases.
In all instances our approach produces significant simplifications with respect to the original proofs. Let us explain in which sense these formulas are special cases
of ours. One of the ingredients of the theory is a (in general, meromorphic)
vector field $V$ on $X$, which is lifted to a first-order differential operator
acting on a holomorphic vector bundle $\cE$. This lift is in general obstructed,
therefore one needs to assume its existence. This
is what we do, and is one of the assumptions of the Carrell-Liebermann formula;
our formula is more general than the latter because we compute the integral of any cocycle in a certain cohomology complex that we associate with $V$ and $\cE$,
while the Carrell-Liebermann formula computes the integral of a polynomial
in the Atiyah class of $\cE$.

A special situation arises when $\cE$ is the holomorphic tangent
bundle $\Theta_X$. If the vector field $V$ is holomorphic,
it always has a   lift to a differential operator on $\Theta_X$,
namely, the Lie derivative $\Lie_X$. Thus one gets
a generalization of Bott formula \cite{Bott67}, or, if one integrates
a polynomial in the Chern classes of $X$, Bott's formula itself.

If $V$ is meromorphic, i.e., it is a section of $\Theta_X\otimes \mathcal L$ for
some effective line bundle $\mathcal L$,  in general it does not lift. There are two ways out of this problem.
Either one just treats the special case when the obstruction
vanishes; this produces a particular case of our
localization formula  \eqref{locatimero}. Or, one 
realizes that even if  $V$ does not lift,
it actually defines a residue (as noticed by Carrell \cite{Carr2}, 
see also section \ref{BB} in this paper). However, this 
 residue
does not compute a characteristic number of the tangent bundle $\Theta_X$, 
but rather a characteristic number of the virtual bundle $\Theta_X-\mathcal L^\ast$.
Thus one obtains  Baum-Bott's formula.

Further investigations along this line would naturally lead to consider generalizations
of the holomorphic Lefschetz formulas and applications to Courant algebroids,
e.g., in connection to generalized complex geometry (cf.~\cite{Li}).

\smallskip

{\bf Acknowledgments.} This paper was drafted while the first author
was visiting IMPA in Rio de Janeiro, and was finalized while he was holding a   visiting position at IHES in Paris. He gratefully acknowledges support and    hospitality   from both institutions.  The second author thanks LPTM of the Cergy-Pontoise University for hospitality during his CNRS delegation. We thank Professors Yvette Kosmann-Schwarzbach and Alexei Rosly for their interest in this work.  

\bigskip

\section{Twisted holomorphic equivariant cohomology}\label{atiyahsec}

Let $X$ be an $n$-dimensional compact   complex manifold. We shall denote by $\Theta_X$ its holomorphic tangent bundle and by $T_X$ its tangent bundle when $X$ is regarded as a $2n$-dimensional smooth differentiable manifold. $\Omega^i_X$ will denote the bundle of holomorphic $i$-forms on $X$, while $\Omega^i_{X,\C}$ will denote the bundle
of complex-valued smooth $i$-forms, and $\Omega^{p,q}_X$ the bundle of
forms of type $(p,q)$. (In general, we shall not use a different notation for a bundle
and its sheaf of sections.) The symbol $\Gamma$ will denote global sections.

We shall not use any Lie algebroid theory here but one should note that our constructions are quite heavily motivated and suggested by that theory \cite{ELW99,Hue,algloc}.

\subsection{The Atiyah algebroid}
If $\cE$ is a (rank $r$) holomorphic vector bundle on $X$ we shall denote by $\ati$ the bundle of first order differential operators on $\cE$ with scalar symbol. $\ati$ sits inside
an exact sequence of sheaves of $\cO_X$-modules
\begin{equation} \label{atiyah} 0 \to  \End(\cE) \to \ati \xrightarrow{\sigma } \Theta_X \to 0
\end{equation}
where   $\sigma$  is the symbol map. $\ati$ is actually an example of a {\em holomorphic  
Lie algebroid} \cite{ELW99,Hue}, with bracket defined on its   sections by the   commutator of differential operators, and the symbol map $\sigma$ playing the role of anchor.

The class in $a(\cE)\in\Ext^1(\Theta_X,\End(\cE))$ defining the extension \eqref{atiyah}
is called the {\em Atiyah
class} of $\cE$, and it is the obstruction to the existence of holomorphic connections
on $\cE$ \cite{Atiyah}. 

We shall assume throughout that a global holomorphic vector field $V$ on $X$
has been fixed, which admits a lift 
$\tilde V$  in $\Gamma(\ati)$. The pair $(\cE,\tilde V)$ is called
an {\em equivariant holomorphic vector bundle.}

\subsection{The cohomology complex}\label{eqholcohom}
We want to define a ``holomorphic equivariant'' cohomology complex
associated with the Atiyah algebroid $\ati$.  We set
\begin{gather}D_\cE = [\CC_X\otimes_{\cO_X}\ati] \oplus T^{0,1}_X \nonumber \\[3pt]
Q^{p,q}_{\cE} = \Lambda^{p}\ati^\ast\otimes_{\cO_X}\Omega^{0,q}_X \nonumber \\[3pt]
Q^k_{\cE} = \Lambda^k D_{\cE}^\ast = \bigoplus_{p+q=k} Q^{p,q}_{\cE} \label{deco}
\,.\end{gather}

Note that the line bundle $\det(D_{\cE})$ is isomorphic to the determinant
of the complexified tangent bundle $T_X\otimes\C$. Moreover,   combining the symbol map with the inner product one defines a  morphism
$$q\colon \Lambda^kD_{\cE}^\ast\otimes \det(D_{\cE}) \to \Lambda^{2n-k+r^2}(T_X\otimes \C).$$
We define a morphism  $p\colon Q_{\cE}^k\to\Omega^{k-r^2}_{X,\C}$ by setting\footnote{Here we use the symbol $\lrcorner$ to denote inner product, however for typographical reasons later on we shall also use the symbol $\imath$.}
$$ p(\psi)=(-1)^k\,q(\psi\otimes \alpha)\lrcorner \eta $$
if $\psi$ is a section of $\Lambda^kD_{\cE}^\ast $, and $\alpha\in \Gamma(\det(D_{\cE}))$ and $\eta$  in $\Omega^{2n}_{X,\C}(X)$ are such that $\eta(\alpha)=1$.
 This map is an  isomorphism when $\cE=\{0\}$. By using this map we can integrate sections of $ Q_{\cE}^\bullet$, i.e., if $\gamma\in  Q_{\cE}^\bullet(X)$, and $X$
 is compact,  by $\int_X\gamma$ we mean $\int_Xp(\gamma)$.

We define an ``equivariant'' complex
$$\Q^\bullet_{\cE}=\C[t] \otimes_{\C} Q^\bullet_{\cE} (X)$$
with the usual equivariant grading
$$\deg({\mathcal P}\otimes\beta) = 2\deg({\mathcal P})+\deg\beta$$
if $\mathcal P$ is a  monomial in $t$ and $\beta\in Q^\bullet_{\cE}$.
We also define a differential
 $$\tilde\delta_V=\bar\partial_{\ati} - t\,\imath_{\tilde V}$$
 where $\bar\partial_{\ati} $
  is  any of the Cauchy-Riemann operators
 of the holomorphic bundles $\Lambda^{k}\ati^\ast$.
 We have $\tilde\delta_V\colon \Q^\bullet_{\cE}\to\Q^{\bullet+1}_{\cE}$, and
 an easy computation shows that $\tilde\delta_V^2=0$, so that
 $(\Q_{\cE}^\bullet,\tilde\delta_V)$ is a cohomology complex.
 We denote its cohomology by  $\mathfrak H_{\tilde V}^\bullet(\cE)$. 

There is a relation between the complex $\mathfrak Q_0^\bullet$ that one obtains by setting
$\cE=\{0\}$ in
$\mathfrak Q_\cE^\bullet$
and  Liu's holomorphic equivariant de Rham complex  \cite{Liu}. Liu's complex is defined  letting  
\begin{equation}\label{LiuComplex}A^{(k)} = \bigoplus_{q-p=k}\Omega^{p,q}(X)\end{equation}
 with a differential
 $$ \delta_t = \bar\partial - t\,\imath_V$$
 for some value of $t$. One has, for every $t$,   cohomology complexes $(A^{(k)},\delta_t)$,
 where the index $k$ ranges from $-n$ to $n$. Liu shows that the corresponding cohomology groups
 $H^{(k)}_t(X)$ are independent of $t$, provided $t\ne 0$.  We shall denote them by
$H^\bullet_{\mbox{\footnotesize\rm Liu}}(X)$.
An explicit computation
 shows the following relation. Let us denote by $\mathfrak H_{ V}^\bullet(X)$ the
 cohomology groups $\mathfrak H_{\tilde V}^\bullet(\cE)$ corresponding to the case
 $\cE=\{0\}$. Moreover, let  $\tilde {\mathfrak H}^{2n+k}_V(X)$ be the subspace of $\mathfrak H^{2n+k}_V(X)$
 formed by classes of cocycles in the subspace $\oplus_p[\C[t]\otimes\Omega^{p,p+k}(X)]$
 of $\oplus_j \mathfrak Q^j_0$.

 \begin{prop} \label{liu} For every $k=-n,\dots,n$, and every $t^\ast\in\C^\ast$,  the cohomology group $H^{(k)}_{\mbox{\footnotesize\rm Liu}}(X)$
 is isomorphic as a $\C$-vector space  to the subspace of $\tilde { \mathfrak H}^{2n+k}_V(X)$
 obtained by setting $t=t^\ast$.
 \end{prop}
\begin{proof} Let us realize $H^{(k)}_{\mbox{\footnotesize\rm Liu}}(X)$ as $H^{(k)}_1(X)$.
A class $[\omega]$  in this group is represented by an element
$$ \omega=\sum_{p=0}^n\omega^{p,p+k} $$
(where $\omega^{p,p+k} \in\Omega^{p,p+k}(X)$) satisfying $(\bar\partial-\,\imath_V)\omega=0$, i.e.,
\begin{equation}\label{Liucocycle}\bar\partial \omega^{p,p+k} = \imath_V \omega^{p+1,p+k+1}\quad\text{for}\quad p=0,\dots,n.\end{equation}
If we define the element in $\mathfrak Q_0^\bullet$
$$\xi = \sum_{p=0}^n t^{n-p} \omega^{p,p+k}$$
then the condition \eqref{Liucocycle} is equivalent to $\tilde\delta_V\xi=0$. Thus we obtain
a  class in $\tilde {\mathfrak H}^{2n+k}_V(X)$.  Conversely, given a cocycle representing a class
in this space, the previous computation shows that  evaluating it at $t=1$ we get
a class in  $H^{(k)}_1(X)$. The same is true for any $t^\ast\in\C^\ast$.
\end{proof}
In particular $H^{(0)}_{\mbox{\footnotesize\rm Liu}}(X)$ is isomorphic to
the subspace of $\oplus_{j} \mathfrak H^j_V(X)$ generated over $\C[t]$ by the  classes in $\oplus_p\Omega^{p,p}(X)$.

\bigskip
\section{Localization}
We turn now to the construction of a localization formula for the  holomorphic equivariant cohomology $\mathfrak H_{\tilde V}^\bullet(\cE)$ we introduced in Section \ref{eqholcohom}. Since we want to cover also the case when the vector field $V$ has nonsimple zeroes we need to introduce the notion of  {\em Grothendieck residue} \cite{Hart,Carr}. 
\subsection{The Grothendieck residue} We recall here the definition as given in \cite{Carr}. Let us start from the situation where we have a line bundle $\mathcal L$ on $\C^n$, and $n$ sections $a_1,\dots,a_n$ of $\mathcal L$ that have a common isolated zero at $0$. Let $s$ be a section of $\mathcal L^n\otimes\Omega^n_{\C^n}$, and let $D$ be a disc in $\C$, centred at 0, such that
 $D^n$ does not contain zeroes of $a_1,\dots,a_n$ other than 0,
 and the product $a_1\cdots a_n$ does not vanish on 
 $\partial D \times\cdots\times \partial D$.
 Then one sets
\begin{equation}\label{res}\operatorname{Res}\left(\begin{array}{cc} s \\ a_1 \dots a_n\end{array}\right)
 =\frac1{(2\pi i )^n} \int_{\partial D \times\cdots\times \partial D}\ \frac{s}{a_1\cdots a_n}
 \end{equation} Now, given a line bundle $\mathcal L$ on $X$, let $V$ be a global section of 
 $\Theta_X\otimes \mathcal L$ that has isolated, but possibly degenerate, zeroes.
 If $x_0$ is one of such zeroes, and $s$ is a section of $\mathcal L^n\otimes\Omega^n_{X}$ on a neighbourhood of $x_0$, let us choose holomorphic coordinates $(z_1,\dots,z_n)$ centred in $x_0$, and write $V = \sum_i a_i\,\frac{\partial}{\partial z_i}$
 locally around $x_0$. We can define
\begin{equation}\label{res1}\operatorname{Res}_{V,x_0} (s)= \operatorname{Res}\left(\begin{array}{cc} s \\ a_1 \dots a_n\end{array}\right).\end{equation}
 In the particular case when the zero of $V$ at $x_0$ is nondegenerate, so that the
 Jacobian determinant $J(a_1 \dots a_n)_{x_0}$ of the partial derivatives of the components
 $a_i$  at $x_0$ does not vanish, one can write
\begin{equation}\label{res2}\operatorname{Res}_{V,x_0} (s) = \frac{s(x_0)}{J(a_1 \dots a_n)_{x_0}}.\end{equation}
An algorithm to get an explicit expression for the Grothendieck residue in the general case (i.e., when $V$ has degenerate zeroes) is given in \cite{BaumBott}.

The residue \eqref{res1} is independent of the choices of the coordinates, and defines a morphism 
$$\operatorname{Res}_{V,Z} \colon H^0(Z,\mathcal L^n) \to \C$$
 where $Z$ is the closed 0-dimensional subscheme of $X$ given by the zeroes of $V$
 (more precisely, $Z$ is the subscheme associated with the sheaf of ideals
 $\mathcal  I_Z = \imath_V (\Omega_X^1\otimes\mathcal L^\ast)\subset \cO_X$).
 
 The previous discussion implies the following result.
  \begin{lemma} If $V$ has isolated nondegenerate zeroes $\{x_j\}$, one has
  $$\operatorname{Res}_{V,Z}(s) = \sum_j \frac{s(x_j)}{\det \mathbb L_{V,j}}$$
where $${\bbL} _{ V,j} \colon \Theta _{X,x_j}\to  (\Theta _{X}\otimes \mathcal L)_{x_j}$$
is the linear transformation defined as 
${\bbL} _{ V,j} (u) = [V,u]$.
\label{easyresidue}
\qed\end{lemma}
\noindent Note that $\det \mathbb L_{V,j}\in \mathcal L^n_{x_j}$.

\subsection{The localization formula} We consider at first a 
 localization formula in the case when  the line bundle involved in the definition of the Grothendieck residue is trivial, $\mathcal L = \cO_X$.

\begin{thm} \label{locati} Let $X$ be an $n$-dimensional compact   complex manifold,
$\cE$ a holomorphic vector bundle on $X$, and $V$ a holomorphic vector
field on $X$, which lifts to a section of $\ati$, and has isolated   zeroes.
If $\gamma\in\Q^\bullet_\cE$ is such that $\tilde\delta_V\gamma=0$, we have
\begin{equation}\label{locatif}
\int_X \gamma (t)=\left(\frac{2\pi\, i}{t}\right)^n \operatorname{Res}_{V,Z}(p(\gamma)_0(t))\,.\end{equation}
where $Z$ is the scheme of the zeroes of $V$.
\end{thm}

\begin{remark} If the zeroes $\{x_j\}$ are nondegenerate, in view of Lemma \ref{easyresidue} the localization formula
may be written as 
\begin{equation*}\label{locatif0}
\int_X \gamma (t)=(2\pi i)^n\sum_{j} \frac{p(\gamma)_0(x_j)(t)}{t^n \, \det  \bbL_{V,j}}\end{equation*}
where ${\bbL} _{ V,j} \colon \Theta _{X,x_j}\to  \Theta _{X,x_j}$. In this form this formula is basically equivalent to Feng-Ma's localization formula \cite{Feng-Ma05}.
\end{remark}
\begin{proof} 
The proof of the localization formula in equivariant cohomology as given in \cite{BGV92}
may be adapted to provide an easy proof of formula \eqref{locatif}.

The map $p\colon Q^k_\cE \to \Omega_{X,\C}^{k-r^2}$ may be written ---
with respect to the decomposition \eqref{deco} ---
as $p=\sum_{i+j=k} p_i\otimes \imath_j$, where $p_i\colon \CC_X\otimes \Lambda^i\ati^\ast\to\Omega^{i-r^2,0}_{X}$, and $\imath_j\colon\Omega^{0,j}_X \to \Omega^{0,j}_X $ is the identity map. This implies
the identity
\begin{equation}\label{commu}
 p \circ \bar\partial_{D_\cE} =  \bar\partial \circ p.
 \end{equation}
This reduces the proof of the localization formula to the case $\cE=\{0\}$.
Let us denote by $\delta_V=\bar\partial - t \, \imath_V$ the ``holomorphic equivariant'' differential for the complex $\mathfrak Q_0^\bullet = \C[t]\otimes
\Omega^\bullet_{X,\C}$.

Choose an hermitian
metric $g$ on $X$, and denoting by $\bar g\colon T^{0,1}_X \to \Omega^{1,0}_X$ the
corresponding homomorphism, let
$\theta = \bar g(\bar V)/\Vert V\Vert^2$ (so $\theta $ is of type (1,0)). 
Using the injection $\Omega^1_{X,\C} \to D^\ast_{\cE}$, $\theta$ may be regarded
as a cochain in $\mathfrak Q^1_{\cE}$.
Note that $(\delta_V \theta)_0= -t $
so that $\delta_V\theta$ is invertible in the ring of differential forms away from the zeroes
of $V$ if $t\ne 0$, and one has
$$\delta_V \left(\frac{\theta}{\delta_V\theta}\right) = 1 \,.$$
We set
$$ \alpha =\frac{\theta}{\delta_V\theta},\qquad
\tilde \alpha = \alpha \wedge \gamma, $$
so that 
$ \delta_V  \tilde \alpha = \gamma$ (again, away from the zeroes of $V$).
If we set
$$\gamma = \sum_{k=0}^n t^k\, \gamma_{2n-k},\qquad
\alpha = \sum_{k=0}^{n-1} \alpha_{2k+1}$$
with
$$\alpha_{2k+1} = \theta \left(\frac1{\tilde\delta\theta}\right)_{(2k)}=
-( -t)^{-k-1}\theta\wedge (\bar\partial\theta)^k$$
then away from the zeroes of $V$ one has
$$\gamma_{2n} = \bar\partial w $$
with
\begin{equation}\label{w} w = -  \sum_{k=0}^n \gamma_{2n-2k}\wedge \theta\wedge (\bar\partial\theta)^{k-1}\,.\end{equation}
Note that $p(\gamma)_0=\gamma_0$. Denoting by $x_j$ the zeroes of $V$, let $B_j(\epsilon)$ be a ball of radius $\epsilon$ (with respect to the hermitian metric we chose) with center
$x_j$. Then by Stokes theorem and equation \eqref{w}
$$\int_X\gamma = - \sum_{j} \lim_{\epsilon\to 0} \int_{\partial B_j(\epsilon)} w
= -\sum_{j} \lim_{\epsilon\to 0} \int_{B_j(\epsilon)} \bar\partial w 
= \sum_{j}  \lim_{\epsilon\to 0}   \int_{B_j(\epsilon)} \gamma_0 (\bar\partial\theta)^n\,.$$
Let $Z_j$ be the component of $Z$ supported on $x_j$.  
Since we may replace each ball $B_j(\epsilon)$ with a polydisk
of radius, say, $\epsilon$, comparing with the definition 
\eqref{res} of the Grothendieck residue we get 
$$ \lim_{\epsilon\to 0}   \int_{B_j(\epsilon)}\gamma_0(\bar\partial\theta)^n =(2\pi i )^n \operatorname{Res}_{V,Z_j}(\gamma_0)$$
and therefore we obtain
the expression
of the right-hand side of the localization formula. 
The reader may also compare with the proof of formula (9) in \cite{Carr2}.
 \end{proof}
\begin{remark}\label{redLiu} If we take $\cE=\{0\}$ and consider only classes $\gamma$ in  $\C[t]\otimes[\oplus_p\Omega^{p,p}(X)]$,  
this localization formula reduces to Liu's   formula 
\cite[Thm.~1.6]{Liu}.  
\label{liuloc}
\end{remark}

\begin{ex}\label{exDH} If $\cE=\cO_X$,  the exact sequence \eqref{atiyah} splits, and one has 
$$Q_{\cO_X}^k = \Omega_{X,\C}^k \oplus \Omega_{X,\C}^{k-1}\,.$$ 
The morphism $p$ maps $Q_{\cO_X}^{k}$ to $\Omega_{X,\C}^{k-1}$.
An element in $\Q_{\cO_X}^{2n+1}$ (where $n=\dim_\C X)$ has the form
$$\sum_{k=0}^n(\omega_{2n+1-2k} + \eta_{2n+1-2k})t^k$$
where $\omega_k\in \Omega_{X,\C}^k(X) $ and $\eta_k\in \Omega_{X,\C}^{k-1} (X)$.

Let $V$ be a holomorphic vector field on $X$ with isolated nondegenerate zeroes.
If $\omega_2\in \Omega_{X,\C}^2(X)$ and $f\in \CC(X)$ are such that $\bar\partial\omega_2=0$, $\iota_V\omega_2=\bar\partial f$, then $\omega = (\omega_2+t\,f)^n\in \Q_{\cO_X}^{2n+1}$ is a cocycle,
i.e., $\hat\delta_V\omega=0$. The localization formula \eqref{locatif} gives
\begin{equation}\label{DH1} \int_X (\omega_2)^ n = (-2\pi i)^n\sum_{\ell} \frac{f(x_\ell)}{J_{\ell}}\end{equation}
where $x_\ell$ are the zeroes of $V$, and $J_\ell$ are the Jacobian determinants of the components of $V$ at those zeroes. Equation \eqref{DH1} is a kind of complex Duistermaat-Heckman formula (but $\omega_2$ need not be nondegenerate).

The existence of the holomorphic vector field $V$ with isolated zeroes, and of nontrivial $\omega_2$, $f$ satisfying $\iota_V\omega_2=\bar\partial f$, puts conditions on the variety $X$. If  $X=\mathbb P^1$, such a vector field obviously exists, and a simple check, using the Fredholm alternative for the Laplacian on functions on $\mathbb P^1$, shows that for any choice of $\omega_2$, there exists a function $f$ satisfying   $\iota_V\omega_2=\bar\partial f$. ($\bar\partial\omega_2=0$ is automatic.) \end{ex}

\subsection{Moment maps}  We develop some further techniques, in particular
we introduce an appropriate notion of {\em moment map}. 
Let $K$ be the curvature of the Chern connection $\nabla$ of the pair $(\cE,h)$
(where $h$ is an hermitian metric on $\cE$), i.e., the unique connection
on $\cE$ which is compatible both with the holomorphic structure of $\cE$ and the
metric $h$. Let $V$ be a holomorphic vector field on $X$, with   isolated zeroes, and let $\tilde V$ be a lift of $V$. The  {\em moment map}
$\mu$ is the $C^\infty$ endomorphism of $\cE$ given by
\begin{equation}\label{moment} \mu = \tilde V - \nabla_V\,. \end{equation}

\begin{lemma} The moment map $\mu$ enjoys the following properties:
\begin{enumerate} \item 
$\bar\partial_{End(\cE)} \mu =\imath_V K$, where $K$ is the curvature of $\nabla$; \item
$\mu(x_j) = \mathbb L_{\tilde V,j}$ for all zeroes $x_j$ of $V$, where
${\bbL} _{\tilde V,j}\colon
\cE_{x_j}\to \cE_{x_j}$ is the endomorphism induced by the differential
operator $\tilde V$ 
 (note that at the zeroes
$x_j$ of $V$, the differential operator $\tilde V$   has degree 0). 
\end{enumerate}
\end{lemma}
\begin{proof} The first claim is proved as in \cite{Bott67} by the following chain of identities. Let $u$ be a vector field of type $(0,1)$ on $X$, and
$s$ a section of $\cE$.
\begin{eqnarray*} \langle( \bar\partial_{End(\cE)}\mu)(s),u\rangle &=& \langle \bar\partial_\cE(\mu(s)),u\rangle = - \nabla_u\nabla_V s \\
&=& \nabla_V\nabla_u s - \nabla_u\nabla_V s - \nabla_{[V,u] } s 
= \langle K(s),V\wedge u\rangle \\ &=& \langle \iota_VK(s),u\rangle\,.
\end{eqnarray*}
The second statement is evident. \end{proof}
We define the equivariant curvature of the Chern connection as
$\tilde K = K+t\,\mu$. By using the connection $\nabla$ to split the exact sequence
\begin{equation}\label{split} 0 \to \End(\cE)\otimes\CC_X \to D_\cE \to T_X\otimes{\C} \to 0 \end{equation}
we may regard $\tilde K$ as an element in $\mathfrak Q^3_\cE$.  

\begin{lemma}\label{closed} For every $m\ge 0$, the 
the cochain $\tr(\tilde K^m) \in\mathfrak Q^{2m}_\cE$ 
 is closed, i.e., $
 \tilde\delta_V\tr(\tilde K^m)=0$.
 \end{lemma} 
 \begin{proof} 
After splitting the exact sequence \eqref{split}, the complex $\mathfrak Q^\bullet_\cE$
acquires a bigrading; let $\Omega^1_{X,\C}$ have bidegree (1,0), and $\End(\cE)$
bidegree (0,1).  
Then $\tr(\tilde K^m)$ has bidegree $(2m,0)$, while 
$\bar\partial_{\ati}$ has bidegree (1,0), and $\imath_{\tilde V}$ has
a piece of bidegree $(-1,0)$ (which is basically $\imath_V$) and a piece of bidegree $(0,-1)$.
Then one has
\begin{multline*} \tilde\delta_V \tr(\tilde K^m) 
= ( \bar\partial_{\ati} - t\,\imath_{\tilde V})   \tr(\tilde K^m) =
 ( \bar\partial_{\ati} - t\,\imath_{V}) \tr(\tilde K^m) =  \\
 m \tr (\tilde K^{m-1}  ( \bar\partial_{\ati} -\,\imath_{V}) \tilde K) =
 m \tr (\tilde K^{m-1} (\bar\partial_{\ati} K+ t\, \bar\partial_{\ati}\mu - t\,\imath_V K) )=0
\end{multline*}
the equality to zero of the last expression being due to the Bianchi identities and
the definition of the moment map.
 \end{proof} 
 
\begin{ex} \label{simpleCL} Assume   that $\cE$ is a line bundle $\mathcal M$, and define 
$$\gamma(t) = \left(\frac{i}{2\pi} \tilde K\right)^n$$
where $n=\dim X$. Then $\gamma$ is a cocycle in $\mathfrak Q^\bullet_{\cE}$, and moreover
we have
$$p(\gamma)_0(t) = \left(\frac{it}{2\pi}\mu\right)^n\,.$$
Assume moreover that $V$ has isolated nondegenerate zeroes. Theorem 3.2 yields
\begin{equation}\label{eqexample}\int_X c_1(\mathcal M)^n = \sum_{j} \frac{  (c_j)^n}{\det(\mathbb L_{V,j})}.\end{equation}
Here $c_1(\mathcal M)$ is a $(1,1)$-form on $X$ representing the first Chern class
of $\mathcal M$, and $c_j$ are complex numbers such that
$\mathbb L_{\tilde V,j}(z)=c_j\,z$.
This equation is a special case of the Carrell-Liebermann localization formula \cite{CL,Carr}; indeed, in the next
Section, we shall see that the formula \eqref{locatif} implies the Carrell-Liebermann formula.

 If $\mathcal M=\cO_X$, we have $a(\cO_X)=0$
and the exact sequence defining the Atiyah algebroid $\mathcal D_{\cO_X}$
splits. Every holomorphic vector field $V$ on $X$ lifts to a section $\tilde V$ of 
$\mathcal D_{\cO_X}$ of the form $c + V$, with $c$ a constant. Formula
\eqref{eqexample} yields the well-known identity \cite{Bott67}
\begin{equation}\label{wellknown}\sum_ j \frac1{\det (\bbL_{V,j} )}=0\,.\end{equation}
\end{ex}

\subsection{Bott's formula}\label{bott} To recover Bott's classical   formula
we set $\cE=\Theta_X$. Any holomorphic vector field $V$ on $X$ has a   lift to a differential operator $\tilde V$ on $\Theta_X$, which is $V$ itself acting as s  Lie derivative. 
At the zeroes $x_j$ of $V$, this defines an endomorphism
$\mathbb L_{\tilde V,x_j} \colon (\Theta_X)_{x_j} \to (\Theta_X)_{x_j}$.
Moreover, in this case $p\colon Q_{\cE}^{2n+r^2} =  Q_{\Theta_X}^{2n+n^2}  \to \Omega^{2n}_{X,\C}$
is an isomorphism. We assume that the zeroes of $V$ are nondegenerate.

Let $P_k$ be the $k$-th invariant elementary polynomial
on the Lie algebra $\operatorname{gl}(n,\C)$, where $n=\dim X$. We use
the normalization of \cite{GH}. We define the equivariant Chern classes of
$X$ as
$$\tilde c_k(X) = P_k\left(\frac{i}{2\pi} \tilde K\right),\quad k=1,\dots,n $$
and the equivariant Chern classes of the endomorphisms
$\mathbb L_{V,j}\colon \Theta_{X,x_j} \to \Theta_{X,x_j}$ as
$$\tilde c_k(x_j) = P_k\left(\frac{i}{2\pi} \tilde \mu(x_j)\right),\quad k=1,\dots,n. $$
Note that $\tilde c_n(x_j) = \det (\bbL_{V,j} ).$ Let $\Phi$ be a polynomial in $n$
variables whose degree in the $i$-th variable is $m_i$, and such that
$\sum_{i=1}^ni\,m_i=n$. By Lemma \ref{closed}, the quantity
$$\Phi(\tilde c_1(X),\dots,\tilde c_n(X))$$
is a cocycle in $\mathfrak Q^\bullet_{\Theta_X}$. The localization formula \eqref{locatif}
in this case reads
$$\int_X\Phi(c_1(X),\dots,c_n(X)) = \sum_j \frac{\Phi(\tilde c_1(x_j),\dots,
\tilde c_n(x_j))}{\tilde c_n(x_j)},$$
i.e., we recover Bott's formula. (One uses the formula in Lemma
\ref{easyresidue}; note that the determinant in the denominator
of the residue is just the Chern class $\tilde c_n(x_j)$.)

\bigskip
\section{The twisted case} 
One can also write a localization formula in the case when $V$ is
a global section of   $\Theta_X\otimes\mathcal L$, where $\mathcal L$ is a holomorphic line bundle. 
The advantage in making such a twist is that if $X$ is projective, then by Serre theorem one can always choose $\mathcal L$
so that there exist sections of $H^0(X,\Theta_X\otimes\mathcal L)$ that may be lifted
to a differential operator  $\cE\to\cE\otimes\mathcal L$ \cite{CL,Carr}.

If $\mathcal L$ is
an effective line bundle, i.e., $\mathcal L \simeq \cO_X(D)$ for
an effective divisor $D$ in $X$, the localization formula we shall obtain contains
Baum-Bott's formula for meromorphic vector fields \cite{BaumBott} and Carell-Liebermann's formula \cite{CL}
as special cases.

\subsection{A new cohomology complex} We define a new cohomology 
complex by setting \footnote{For $\cE=0$ this complex already appeared in \cite{Carr2}.}
\begin{equation}\label{decoL}\widehat Q^k_{\cE} = \bigoplus_{p+q=k}
\widehat Q^{p,q}_{\cE},\qquad \widehat Q^{p,q}_{\cE}=
\Lambda^{p}\ati^\ast\otimes_{\cO_X}\mathcal L^{-p+n+r^2}\otimes_{\cO_X}\Omega^{0,q}_X\end{equation}
and  
$$\widehat\Q^\bullet_{\cE}=\C[t] \otimes_{\C}\widehat Q^\bullet_{\cE} (X)\,,$$
with a grading and differential defined as in the previous case. Moreover, we may define a map 
\begin{multline*} \hat q\colon \bigoplus_{p+q=k} \Lambda^p\ati^\ast\otimes_{\cO_X}\mathcal L^{-p+n+r^2}\otimes_{\cO_X} \det\ati \otimes_{\cO_X} \Omega_X^{0,q} 
\\ \to \bigoplus_{p+q=k} \Lambda^{-p+n+r^2}\Theta_X 
\otimes_{\cO_X}\mathcal L^{-p+n+r^2}\otimes_{\cO_X}  \Omega_X^{0,q} 
\end{multline*}
by setting
$$\hat q(\phi\otimes s\otimes u \otimes\omega) = \sigma(\phi\rfloor u)\otimes s \otimes\omega$$
where $\sigma$ is the anchor $\ati\to\Theta_X$. 
We also define 
$$\hat p\colon \widehat Q^k_{\cE} \to \bigoplus_{p+q=k}  \mathcal L^{-p+n+r^2}
\otimes_{\cO_X} \Omega_X^{p-r^2,q}$$
by letting $\hat p(\psi) =(-1)^k\, \hat q(\psi\otimes\alpha)\rfloor \eta$,
where $\alpha\in \Gamma(\CC_X \otimes_{\cO_X} \det(\Theta_X ))$
and $\eta\in\Omega_X^{n,0}(X)$ are such that $\eta(\alpha)=1$. 
 Note that for
$k=2n+r^2$ one has $\hat p(\psi) \in \Omega_{X,\C}^{2n}(X)$,
and for $k=r^2$, $\hat p(\psi) \in \Gamma( \CC_X \otimes_{\cO_X} \mathcal L^n)$.
Instead of the identity \eqref{commu} we have now
\begin{equation}\label{commuhat} \hat p \circ \bar\partial_{\ati} = (-1)^k \bar\partial_{\mathcal L} \circ\hat p\end{equation}
where $\bar\partial_{\ati}$ is now any of the $\bar\partial$-operators of the sheaves
$\Lambda^p\ati^\ast\otimes\mathcal L^{-p+n+r^2}$,
and $\bar\partial_{\mathcal L}$ is the $\bar\partial$-operator of a suitable power of $\mathcal L$.

Again, we define a complex
$$\widehat\Q^\bullet_{\cE}=\C[t] \otimes_{\C} \widehat Q^\bullet_{\cE} (X)$$
with  a differential
 $$\hat\delta_V=\bar\partial_{\ati} - t\,\imath_{\tilde V}\,.$$

\subsection{Localization formula in the meromorphic case} 
We can now state and proof the localization formula for the meromorphic case.

\begin{thm} \label{locatimero} Let $X$ be an $n$-dimensional connected compact complex manifold,
$\cE$ a holomorphic vector bundle on $X$, 
$\mathcal L$ a holomorphic line bundle, and $V$ a global section of
$\Theta_X\otimes\mathcal L$, which lifts to a section of $\ati\otimes\mathcal L$, and has isolated   zeroes.
If $\gamma\in\widehat \Q^\bullet_\cE$ is such that $\hat\delta_V\gamma=0$, we have
\begin{equation}\label{locatifmero}
\int_X \gamma (t)=\left(\frac{2\pi\, i}{t}\right)^n \operatorname{Res}_{V,Z}(\hat p(\gamma)_0(t))\,.\end{equation}
\end{thm}
\begin{proof}  The proof of Theorem \ref{locati} works also in this case with only minor modifications. By using the identity \eqref{commuhat} we reduce again the proof to the case $\cE=\{0\}$. Note that in this case, $\widehat Q^k_0=\oplus_{p+q=k}\mathcal L^{-p+n}\otimes_{\cO_X}\Omega_X^{p,q}$, so that 
$\widehat Q_0^{2n}=\Omega^{2n}_{X,\C}$, $\widehat Q_0^0=\mathcal L^n$.

We need to fix hermitian metrics both in $X$ and on $\mathcal L$,
and define $\theta$ as the image of the complex conjugate of $V$ in $\Gamma(\Omega^{1,0}_X\otimes\mathcal L)$ via the induced hermitian metric on $\Theta_X\otimes\mathcal L$. The computation in the proof of Theorem \ref{locati} also works in this case; indeed, since the zeroes $x_j$ of  $V$ are isolated, around any $x_j$ we can find a neighbourhood over which $\mathcal L$ trivializes. \end{proof}

\begin{remark} Proposition \ref{liu} and Remark \ref{liuloc} can be extended to the meromorphic case, obtaining the cohomology complex and the localization formula 
(Theorem 6.6) described in Section 6 of \cite{Liu}.
\end{remark}

\subsection{Twisted moment maps}\label{twisted}  In the presence of  the line bundle $\mathcal L$ 
we need to slightly modify the moment map construction of the previous
section. As before, we fix an hermitian metric on $\cE$, and denote
by $K$ the curvature of the corresponding Chern connection. The moment map $\mu$ is 
again defined by equation \eqref{moment}, with $V$ a section of
$\Theta_X\otimes\mathcal L$, and $\cE\to\cE\otimes\mathcal L $ is 
a differential operator which lifts it. This moment map
enjoys the properties
 
\begin{enumerate} \item 
$\bar\partial_{Hom(\cE,\cE\otimes\mathcal L )} \mu =\imath_V K$; \item
$\mu(x_j) = \mathbb L_{\tilde V,j}$ for all zeroes $x_j$ of $V$,
\end{enumerate}
 with $ \mathbb L_{\tilde V,j}$  defined as in Lemma \ref{easyresidue}.

\subsection{Carrell-Liebermann's formula} \label{carrlieb}
The Carrell-Liebermann formula is obtained from the formula \eqref{locatifmero}
just  making a specific choice of the cocycle $\gamma$.
We assume that  $\mathcal L \simeq \cO_X(D)$ for
an effective divisor $D$ in $X$. Thus we have maps
$\mathcal L^{m_1} \to \mathcal L^{m_2}$ whenever $m_1\le m_2$.
This allows  us to regard the traces of the powers of the equivariant curvature $\tilde K = K +t\mu$
as cocycles in $\widehat{\mathfrak Q}_{\cE}^\bullet$.

Let $P$ be an Ad-invariant polynomial on the algebra $\mbox{gl}(r,\C)$ of $r\times r$ complex matrices,
and let $\Phi$ be the polynomial expressing   $P(a(\cE))$
in terms of the Chern classes $c_1(\cE),\dots,c_r(\cE)$, i.e.,
$$ P(a(\cE))= \Phi(c_1(\cE),\dots,c_r(\cE))\,.$$
After setting $\tilde K = K + t\mu$, we define 
 the equivariant Chern forms of
$\cE$ as
$$\tilde c_k(\cE) = P_k \left(\frac{i}{2\pi}\tilde K\right)\,. $$
Then $\Phi(\tilde c_1(\cE),\dots,\tilde c_r(\cE))$
is a cocycle in $\widehat{\mathfrak Q}_{\cE}^\bullet$, and
\begin{equation}\label{int} \int_X P(a(\cE)) = \int_X \Phi(\tilde c_1(\cE),\dots,\tilde c_r(\cE))\,.\end{equation}

Finally,  we note that  evaluating the polynomial $P$ on 
the homomorphism ${\bbL} _{ \tilde V,j} \colon \cE_{x_j}\to  (\cE\otimes \mathcal L)_{x_j}$ we obtain an element  
$P(\mathbb L_{\tilde V,j})\in \ccL^n_{x_j}$, 
and summing over all zeroes of $V$, we obtain a section 
$P(\mathbb L_{\tilde V})\in H^0(Z,\mathcal L^n)$.

The localization formula \eqref{locatifmero} now gives:

\begin{corol}[Carrell-Liebermann's localization formula]
Under the same hypotheses of Theorem \ref{locatimero},
we have
\begin{equation}\label{CaLi} \int_X P(a(\cE)) = (2\pi i)^n \operatorname{Res}_{V,Z}(P(\mathbb L_{\tilde V}))\,.\end{equation}
\end{corol}

\begin{proof} We only need to check that if we make the choice
$\gamma =  \Phi(\tilde c_1(\cE),\dots,\tilde c_r(\cE))$, then
the residue of $p(\gamma)_0$ at $x_j$ coincides with 
$\operatorname{Res}_{V,Z_j}(P(\mathbb L_{\tilde V}))$. This follows from the equality 
\begin{equation*}p(\tilde c_k(\cE))_0 = P_k\left(\frac{i}{2\pi}\mu\right)\,.
\qedhere\end{equation*}
\end{proof}

\begin{rem} We can use Example \ref{exDH} to show  that our localization theorem is indeed more general than the Carrell-Liebermann formula. Indeed, if we take $\cE=\cO_X$,
since the Atiyah class of $\cO_X$ is zero, for every nonzero polynomial $P$ the Carrell-Liebermann formula \eqref{CaLi} simply yields the identity \eqref{wellknown}, i.e.,
$$\sum_\ell \frac{1}{J_\ell} = 0$$
 (cf.~Example \ref{simpleCL}). Note that is implied by equation \eqref{DH1}  taking $\omega_2=0$, $f=1$.
\end{rem} 

\subsection{Baum-Bott formula} \label{BB} Given a line bundle $\mathcal L$ on $X$,
which we assume to be effective, Baum-Bott's meromorphic vector field theorem is a localization formula which expresses the Chern numbers of the virtual bundle ${\Theta_X}-\mathcal L^\ast$ as a sum of residues. One can deduce that formula from our general localization formula \eqref{locatifmero}.  

Let $V$ be   a global holomorphic 
section of ${\Theta_X}\otimes\mathcal L$, that is, a meromorphic vector field
on $X$, having isolated zeroes. In general, it does not lift to a differential operator 
$\tilde V\colon {\Theta_X}  \to {\Theta_X}\otimes\mathcal L$. One can try
to define a lift locally, using a trivialization of $\mathcal L$ and extending the Lie derivative 
by linearity with respect to the coefficients in $\mathcal L$. The local expressions in general do not match, but the mismatch is a multiple of $V$, so that
one obtains a well-defined object in $\operatorname{Hom}( {\Theta_X}_{\vert Z},({\Theta_X}\otimes\mathcal L)_{\vert Z})$, where $Z$ is the zero-cycle of
the zeroes of $V$, as before. If we evaluate a polynomial $P$ on this object,
we obtain a well-defined residue. From one viewpoint, the content of Baum-Bott's meromorphic vector field theorem is that this residue computes a Chern number of the virtual bundle $\Theta_X-\mathcal L^\ast$.

We shall now offer a proof of this fact, basically following \cite{Carr2}. Let
$\{U_i\}$ be an open cover over which $\mathcal L$ trivializes, and for
each $i$, let $t_i$ be a generator of $\mathcal L(U_i)$. Moreover,
let $\{\rho_i\}$ be a (smooth) partition of unity subordinated to $\{U_i\}$, 
and let $\omega_i$ be connection forms, each defined on $U_i$, of the Chern
connection $\nabla$ for $\Theta_X$ given by an hermitian metric on $X$. We define
$$\omega = \sum_i \rho_i(\omega_i - t_i^{-1}\,dt_i\otimes\operatorname{Id}_{\Theta_X}),
\qquad K = \bar\partial\omega,\qquad \mu = \sum_i\rho_i(\tilde V_i-\nabla_V)\,,$$
where $\tilde V_i$ is the local lift of $V$ defined above. 
We have $K\in\Omega^{1,1}(\End(\Theta_X))$, $\mu\in \Gamma(\End(\Theta_X)\otimes\mathcal L)$, and, as in Section \ref{twisted}, $\bar\partial\mu=\imath_V K$. Note that
if $\mathcal L\simeq\cO_X$, one can make choices such that $K$ is the curvature of the Chern connection,
and $\mu$ is a moment map. Now, a direct calculation shows the following.

\begin{lemma} For every $k\ge 1$, the $(k,k)$-form $\tr\bigl(\left(\frac{i}{2\pi}K\right)^k\bigr)$
is closed, and its cohomology class equals the $k$-th Chern character of the
virtual bundle $\Theta_X-\mathcal L^\ast$.\label{virtual}
\end{lemma}

\begin{prop} {\em (Baum-Bott's formula)} Let $V$ be a global section of $\Theta_X\otimes\mathcal L$
having isolated zeroes, and let $\Phi$ be a polynomial in $n$
variables whose degree in the $i$-th variable is $m_i$, and such that
$\sum_{i=1}^ni\,m_i=n$. Denote by $\gamma_i$, $i=1,\dots,n$, the
Chern classes of the virtual bundle $\Theta_X-\mathcal L^\ast$, and
by $\alpha_1^{(j)},\dots,\alpha_n ^{(j)} $ the Chern classes of the linear morphisms
$(\Theta_X)_{x_j} \to (\Theta_X\otimes\mathcal L)_{x_j} $
given by the Jacobian of $V$ at its zeroes.
Then,
$$\int_X\Phi(\gamma_1,\dots,\gamma_n) =(2\pi i)^n \sum_{j} \operatorname{Res}_{V,x_j}(\Phi(\alpha_1^{(j)},\dots,\alpha_n ^{(j)} ))\,.$$
\end{prop} 
\begin{proof} As in Section \ref{bott}, but using the endomorphism-valued forms
$K$ and $\mu$ we have just defined, and taking Lemma \ref{virtual} into account.
\end{proof}

\bigskip

\frenchspacing


\def\cprime{$'$} \def\cprime{$'$} \def\cprime{$'$} \def\cprime{$'$}

\end{document}